\documentclass[12pt,leqno]{amsart}
\usepackage{latexsym,amsmath,amssymb,mathrsfs}
\usepackage{graphicx}

\title[Sobolev space]{A Marcinkiewicz integral type characterization of the Sobolev space}

\author{Piotr Haj{\l}asz and  Zhuomin Liu}

\address{P.\ Haj{\l}asz: Department of Mathematics, University of Pittsburgh, 301
  Thackeray Hall, Pittsburgh, PA 15260, USA, {\tt hajlasz@pitt.edu}}

\address{Z. Liu: Department of Mathematics and Statistics,
University of Jyv\"askyl\"a,
P.O. Box 35 (MaD), FI-40014 Jyv\"askyl\"a, Finland,
{\tt liuzhuomin@hotmail.com}}

\thanks{P.H.\ was supported by NSF grant DMS-1161425. Z.L.\ was supported by the ERC CZ grant LL1203 of the Czech Ministry of Education.}

plus 6pt minus 12pt
\def\eps{\varepsilon}

\def\vi{\varphi}

\def\M{{\mathcal M}}
\def\H{{\mathcal H}}

\newtheorem{theorem}{Theorem}
\newtheorem{lemma}[theorem]{Lemma}

\newtheorem{proposition}[theorem]{Proposition}

\theoremstyle{definition}

\newcommand{\barint}{
\rule[.036in]{.12in}{.009in}\kern-.16in \displaystyle\int }

\newcommand{\barcal}{\mbox{$ \rule[.036in]{.11in}{.007in}\kern-.128in\int $}}

\newcommand{\bbbr}{\mathbb R}

\def\mvint_#1{\mathchoice
          {\mathop{\vrule width 6pt height 3 pt depth -2.5pt
                  \kern -8pt \intop}\nolimits_{\kern -3pt #1}}%
          {\mathop{\vrule width 5pt height 3 pt depth -2.6pt
                  \kern -6pt \intop}\nolimits_{#1}}%
          {\mathop{\vrule width 5pt height 3 pt depth -2.6pt
                  \kern -6pt \intop}\nolimits_{#1}}%
          {\mathop{\vrule width 5pt height 3 pt depth -2.6pt
                  \kern -6pt \intop}\nolimits_{#1}}}

\numberwithin{theorem}{section} \numberwithin{equation}{section}

\begin{document}

\sloppy


\begin{abstract}
In this paper we present a new characterization of the Sobolev space $W^{1,p}$, $1<p<\infty$
which is a higher dimensional version of a result of Waterman \cite{waterman}.
We also provide a new and simplified proof of a recent result
of Alabern, Mateu and Verdera \cite{AMV}. Finally, we generalize the results to the
case of weighted Sobolev spaces with respect to a Muckenhoupt weight.
\end{abstract}

\maketitle

\section{Introduction}
\label{introduction}

In connection with differentiability properties of periodic functions
Marcinkiewicz \cite{marcinkiewicz} introduced the following integral
$$
\nu(f)=\Big(\int_0^{2\pi}|F(x+t)+F(x-t)-2F(x)|^2\frac{dt}{t^3}\Big)^{1/2}\, ,
\quad
\mbox{where} 
\quad
F(x)=\int_0^x f(t)\, dt.
$$
For more details regarding
Marcinkiewicz's results see Vol. II, Chapter XIV, Theorems~5.1 and~5.3 in \cite{zygmund2}.
Marcinkiewicz conjectured that for $1<p<\infty$ there is a constant $C_p>0$ such that
$$
\Vert \nu(f)\Vert_p\leq C_p\Vert f\Vert_p
\quad
\mbox{for $f\in L^p(S^1)$}
$$
and
$$
\Vert f\Vert_p \leq C_p\Vert\nu(f)\Vert_p
\quad
\mbox{for $f\in L^p(S^1)$ such that
$\int_0^{2\pi}f(t)\, dt = 0$.}
$$
The condition in the second inequality that the integral vanishes is necessary, because
for constant functions the right hand side of the inequality equals zero.
The conjecture of Marcinkiewicz was answered in the affirmative by Zygmund
\cite{zygmund}. Later Waterman \cite{waterman} extended the method of Zygmund to the non-periodic case and he proved
\begin{theorem}
\label{T1}
For $1<p<\infty$, there is a constant $C_p\geq 1$ such that
$$
C_p^{-1} \Vert f\Vert_p\leq \Vert\mu(f)\Vert_p \leq C_p\Vert f\Vert_p,
\quad
\mbox{for all $f\in L^p(\bbbr)$,}
$$
where
\begin{equation}
\label{e1}
\mu(f)(x) =\Big(\int_0^\infty |F(x+t)+F(x-t)-2F(x)|^2\frac{dt}{t^3}\Big)^{1/2},
\qquad
F(x)=\int_0^x f(t)\, dt.
\end{equation}
\end{theorem}
Stein \cite{stein} generalized the Marcinkiewicz integral \eqref{e1} to higher dimensions as follows.
Let $\Omega\in L^1(S^{n-1})$ have vanishing integral
\begin{equation}
\label{e1.1}
\int_{S^{n-1}} \Omega(y)\, d\sigma(y)=0.
\end{equation}
The Marcinkiewicz integral of Stein is defined by
\begin{equation}
\label{e1.2}
\mu_\Omega(f)(x)=\Big(\int_0^\infty \Big|\int_{|y|\leq t} \frac{\Omega(y')}{|y|^{n-1}} f(x-y)\, dy\Big|^2\frac{dt}{t^3}\Big)^{1/2},
\quad
\mbox{where $y'=y/|y|$.}
\end{equation}
If $n=1$ and $\Omega(y')={\rm sign}\, y$, we obtain integral \eqref{e1}.
Stein proved in all dimensions that if $\Omega$ is odd, then $\mu_\Omega$ is bounded in $L^p$, $1<p<\infty$, and if 
$\Omega$ is H\"older continuous with exponent $0<\alpha\leq 1$, then $\mu_\Omega$ is bounded in 
$L^p$, $1<p\leq 2$, and is of weak type $(1,1)$. In the odd case the result was obtained as a consequence of the
one dimensional result due to Waterman. The methods used by Stein were quite difficult. 
Later Benedek, Calder\'on and Panzone \cite{BCP} proved the following result
by way of vector valued singular integrals.
\begin{theorem}
\label{T1.1}
If $\Omega\in C^1(S^{n-1})$ satisfies \eqref{e1.1}, then for $1<p<\infty$ there is a constant 
$C=C(n,p)\geq 1$ such that
\begin{equation}
\label{e1.3}
\Vert \mu_\Omega (f)\Vert_p \leq C \Vert f\Vert_p
\quad
\mbox{for $f\in L^p(\bbbr^n)$.}
\end{equation}
\end{theorem}
An optimal condition under which \eqref{e1.3} is satisfied was discovered in \cite{Pan}:
\eqref{e1.3} holds true provided
$\Omega$ satisfies \eqref{e1.1} and $\Omega\in L(\log L)^{1/2}(S^{n-1})$.
A generalization of Theorem~\ref{T1.1} to the case of weighted $L^p_w(\bbbr^n)$ spaces,
where $w\in A_p$ is a Muckenhoupt weight, was obtained by Sato in \cite{sato}, see Theorem~\ref{T11} below.
For recent sharp results, see \cite{duo}.
There has been a tremendous development of the theory of Marcinkiewicz integrals 
(literally hundreds of papers) and 
it is simply not possible to provide relevant references here;  nonetheless the reader will have no problems 
with finding them. 

It turns out that under certain additional assumptions about $\Omega$ we have
\begin{equation}
C^{-1}\Vert f\Vert_p\leq\Vert \mu_\Omega (f)\Vert_p \leq C \Vert f\Vert_p
\quad
\mbox{for $f\in L^p(\bbbr^n)$.}
\end{equation}
The left inequality is obtained from the right one by a duality argument
(see Step~4 in Section~\ref{3} for details).
These assumptions are satisfied for example by $\Omega(y')={\rm sign}\, y$ when $n=1$
and hence Theorem~\ref{T1} follows.

Recall that the Sobolev space $W^{1,p}(\bbbr^n)$ is the space of functions $f\in L^p(\bbbr^n)$
with first order weak derivatives in $L^p(\bbbr^n)$. $W^{1,p}(\bbbr^n)$ is a Banach space
with the norm $\Vert f\Vert_{1,p}=\Vert f\Vert_p+\Vert\nabla f\Vert_p$.
Observe that Theorem~\ref{T1} can be regarded as a characterization of the Sobolev space $W^{1,p}(\bbbr)$.
Indeed, if $f\in W^{1,p}(\bbbr)$, $1<p<\infty$, and
\begin{equation}
\label{e24}
T(f)=\Big(\int_0^\infty |f(x+t)+f(x-t)-2f(x)|^2\frac{dt}{t^3}\Big)^{1/2},
\end{equation}
then $T(f)=\mu(f')$ and hence
\begin{equation}
\label{e1.4}
C_p^{-1}\Vert f'\Vert_p\leq \Vert T(f)\Vert_p\leq C_p\Vert f'\Vert_p.
\end{equation}
It follows from this inequality that $f\in W^{1,p}(\bbbr)$ if and only if $f\in L^p(\bbbr)$ and $T(f)\in L^p(\bbbr)$.
Stein~\cite{stein2}, \cite[p.\ 163]{stein3} generalized this characterization to higher dimensions
as follows 
\begin{theorem}
\label{T2}
$f\in W^{1,p}(\bbbr^n)$, $\frac{2n}{n+1}<p<\infty$ if and only if $f\in L^p(\bbbr^n)$ and
\begin{equation}
\label{e22}
\Big(\int_{\bbbr^n}\frac{|f(x+y)+f(x-y)-2f(x)|^2}{|y|^{n+2}}\, dy\Big)^{1/2}\in L^p(\bbbr^n).
\end{equation}
\end{theorem}
Note that when $n=1$, the integral in \eqref{e22} equals $\sqrt{2}T(f)$
and hence Theorem~\ref{T2} is a natural generalization of the characterization of $W^{1,p}(\bbbr)$ mentioned
above to higher dimensions, but a problem is that in higher dimensions Theorem~\ref{T2} does not cover the
case $1<p\leq 2n/(n+1)$.
Actually Stein proved a more general result that includes a characterization of Bessel potential spaces. 
For other related characterizations of Sobolev and Bessel potential spaces, see for example
\cite{AMV,dorronsoro,nicolau,strichartz,wheeden1,wheeden2}. We will discus the paper \cite{AMV} later on.

One of the aims of this paper is to generalize the characterization \eqref{e1.4} to higher dimensions
in a way that it would be valid for all $1<p<\infty$.
To avoid the limitation on the exponent $p$ in Stein's Theorem~\ref{T2}, we generalize the Marcinkiewicz integral
in a different way. Observe that in dimension one
$$
\frac{f(x+t)+f(x-t)-2f(x)}{2}= \barint_{S(x,t)} f(y)\, d\sigma(y) - f(x)=f_{S(x,t)}-f(x),
$$
where the barred integral denotes the integral average and in our case we take 
the average over the zero dimensional sphere $S(x,t)$. Here and in what follows $f_E$ is used to denote the integral average of $f$ over $E$.
Now for $f\in L^p(\bbbr^n)$ we define
\begin{equation}
\label{e2}
Tf(x)=\Big(\int_0^\infty \Big|f(x)-f_{S(x,t)}\Big|^2\frac{dt}{t^3}\Big)^{1/2}.
\end{equation}
Note that when $n=1$, the definition \eqref{e2} is consistent with \eqref{e24} (up to a constant factor).
One of the main results of this paper reads as follows.
\begin{theorem}
\label{T3}
Suppose that $f\in L^p(\bbbr^n)$, $1<p<\infty$. Then $f\in W^{1,p}(\bbbr^n)$ if and only if
$Tf\in L^p(\bbbr^n)$. Moreover there is a constant $C=C(n,p)\geq 1$ such that
\begin{equation}
\label{e19}
C^{-1}\Vert \nabla f\Vert_p\leq \Vert Tf\Vert_p\leq C\Vert\nabla f\Vert_p.
\end{equation}
\end{theorem}
When $n=1$, Theorem~\ref{T3} is the same as the characterization \eqref{e1.4}.
We will also prove that $Tf$ can be expressed as a Marcinkiewicz  integral \eqref{e1.2} of $\nabla f$.
Let
$$
\phi(x)=\frac{1}{n\omega_n}\, \frac{x}{|x|^n}\, \chi_{B(0,1)}(x)
\quad
\mbox{and}
\quad
\phi_t(x)=t^{-n}\phi(x/t)=\frac{1}{tn\omega_n}\, \frac{x}{|x|^n}\, \chi_{B(0,t)}(x).
$$
Here and in what follows $\omega_n$ is the volume of a unit ball in $\bbbr^n$ and hence 
$n\omega_n$ is the surface area of the sphere $S^{n-1}(0,1)$.
\begin{lemma}
\label{T4}
If $f\in W^{1,1}_{\rm loc}(\bbbr^n)$, then
\begin{equation}
\label{e11}
Tf(x)=\Big(\int_0^\infty \big|\phi_t*\nabla f(x)\big|^2\frac{dt}{t}\Big)^{1/2}=
\frac{1}{n\omega_n}\Big(\int_0^\infty
\Big|\int_{|y|\leq t} \frac{y}{|y|^n}\cdot\nabla f(x-y)\, dy\Big|^2\frac{dt}{t^3}\Big)^{1/2}
\end{equation}
for almost all $x\in\bbbr^n$.
\end{lemma}
The functions $\phi_t$ and $\nabla f$ take values in $\bbbr^n$ and $\phi_t*\nabla f$ is defined as 
the integral of the scalar product
$$
\phi_t*\nabla f(x)=\int_{\bbbr^n}\phi_t(y)\cdot\nabla f(x-y)\, dy.
$$
The inequality $\Vert Tf\Vert_p\leq C\Vert\nabla f\Vert_p$ follows directly from 
Lemma~\ref{T4} and Theorem~\ref{T1.1}.
The proof of the reverse inequality $\Vert \nabla f\Vert_p\leq C\Vert Tf\Vert_p$ will be obtained by a 
standard duality argument;
see Section~\ref{3}, Steps~4 and~5.

Theorem~\ref{T1} can be regarded as a characterization of the Sobolev space $W^{1,p}(\bbbr)$, see \eqref{e1.4}
and a comment that follows. Theorems~\ref{T1.1} and~\ref{T2} are higher dimensional generalizations of
Theorem~\ref{T1}. However, in higher dimensions Theorem~\ref{T1.1} cannot be interpreted as a characterization
of $W^{1,p}(\bbbr^n)$ and in Theorem~\ref{T2} we can characterize $W^{1,p}(\bbbr^n)$ but only for 
$p>2n/(n+1)$. From this perspective Theorem~\ref{T3} is a more natural generalization on Theorem~\ref{T1} to higher
dimensions: it works for all $1<p<\infty$ and it gives a characterization of $W^{1,p}(\bbbr^n)$ in terms of the Marcinkiewicz
integral of the gradient \eqref{e11}, just like the characterization \eqref{e1.4} in dimension $n=1$.

Theorem~\ref{T3} is related to a recent characterization of $W^{1,p}(\bbbr^n)$ due to
Alabern, Mateu and Verdera \cite[Theorem~1]{AMV} where instead of subtracting averages over spheres we
subtract averages over balls.
\begin{theorem}
\label{T3.5}
Suppose that $f\in L^p(\bbbr^n)$, $1<p<\infty$. Then $f\in W^{1,p}(\bbbr^n)$ if and only if
$Sf\in L^p(\bbbr^n)$, where
$$
Sf(x) = \Big(\int_0^\infty |f(x)-f_{B(x,t)}|^2\, \frac{dt}{t^3}\Big)^{1/2}.
$$
Moreover there is a constant $C=C(n,p)\geq 1$ such that
\begin{equation}
\label{e20}
C^{-1}\Vert \nabla f\Vert_p\leq \Vert Sf\Vert_p\leq C\Vert\nabla f\Vert_p.
\end{equation}
\end{theorem}
We will provide a new, and on a technical side much
simpler, proof of this result based on the following representation formula.
Let
$$
\psi(x)=\frac{1}{n\omega_n}\Big(\frac{x}{|x|^n}-x\Big)\chi_{B(0,1)}(x)
\quad
\mbox{and}
\quad
\psi_t(x)=t^{-n}\psi(x/t).
$$
\begin{lemma}
\label{T3.51}
If $f\in W^{1,1}_{\rm loc}(\bbbr^n)$, then
$$
Sf(x)=
\Big(\int_0^\infty |\psi_t*\nabla f(x)|^2\frac{dt}{t}\Big)^{1/2} 
=
\frac{1}{n\omega_n}\Big(\int_0^\infty\Big|\int_{|y|\leq t}
\Big(\frac{y}{|y|^n}-\frac{y}{t^n}\Big)\cdot\nabla f(x-y)\, dy\Big|^2\frac{dt}{t^3}\Big)^{1/2}
$$
for almost all $x\in\bbbr^n$.
\end{lemma}
The original proof of Theorem~\ref{T3.5} is based on vector valued singular integrals and the main
technical difficulty is a verification that a suitable vector valued integral operator satisfies the H\"ormander condition.
Our approach is much simpler as we will show that the inequality $\Vert Sf\Vert_p\leq C\Vert\nabla f\Vert_p$ in
Theorem~\ref{T3.5} is a direct consequence of Lemma~\ref{T3.51}, Theorem~\ref{T1.1},
and the following classical result in the Littlewood-Paley theory due to 
Benedek, Calder\'on and Panzone \cite{BCP}, \cite[Chapter~XII, Theorem~3.5]{torchinsky}.
(Alternatively one can use a result of Sato, Theorem~\ref{T11}, in place of Theorems~\ref{T1.1} and~\ref{T6}, see \eqref{efef}.)
\begin{theorem}
\label{T6}
Let $\phi\in L^1(\bbbr^n)$ be such that
\begin{equation}
\label{e4}
\int_{\bbbr^n}\phi(x)\, dx = 0.
\end{equation}
Assume that there are constants $C,\alpha>0$ such that
\begin{equation}
\label{e5}
|\phi(x)|\leq C(1+|x|)^{-n-\alpha},
\quad
x\in\bbbr^n
\end{equation}
and
\begin{equation}
\label{e6}
\int_{\bbbr^n}|\phi(x+h)-\phi(x)|\, dx \leq C|h|^\alpha,
\quad
h\in\bbbr^n.
\end{equation}
Let $\phi_t(x)=t^{-n}\phi(x/t)$. Then the operator
$$
G f(x) =
\Big(\int_0^\infty |\phi_t*f(x)|^2\, \frac{dt}{t}\Big)^{1/2}
$$
is bounded in $L^p$, $1<p<\infty$ and of weak type $(1,1)$.
\end{theorem}

The main result of \cite[Theorem~3]{AMV} is actually more general since it also covers the case of higher order
derivatives and the case of Bessel potential spaces. It is possible to modify Theorem~\ref{T3}
in a way that it would
cover the case of higher order derivatives, but we decided to restrict to the case of the first order
derivatives for the sake of simplicity.

It is interesting to point out that the functions $Sf$ and $Tf$ satisfy the following 
pointwise inequality. 
\begin{proposition}
\label{T3.6}
If $f\in W^{1,p}(\bbbr^n)$, $1\leq p\leq \infty$, then
\begin{equation}
\label{e3}
Sf(x)\leq \frac{n}{n+2} Tf(x)
\quad
\mbox{a.e.}
\end{equation}
\end{proposition}
Actually
it follows from the proof that \eqref{e3} holds true under a weaker assumption that
$f\in W^{1,1}_{\rm loc}(\bbbr^n)$ is such that
$$
\lim_{t\to \infty} \frac{1}{t}\, \barint_{B(x,t)} |f(y)|\, dy = 0.
$$
As we pointed out, the inequality $\Vert Tf\Vert_p\leq C\Vert\nabla f\Vert_p$ is a direct
consequence of Theorem~\ref{T1.1} and the elementary formula \eqref{e11}.
This combined with Proposition~\ref{T3.6} proves also the inequality $\Vert Sf\Vert_p\leq C\Vert\nabla f\Vert_p$,
but the proof of the reverse inequality $\Vert \nabla f\Vert_p\leq C\Vert Sf\Vert_p$ cannot be directly concluded from
Theorem~\ref{T3} and Proposition~\ref{T3.6}. To prove the reverse inequality we will use 
Lemma~\ref{T3.51} instead of Proposition~\ref{T3.6}.
For this reason we will prove Theorem~\ref{T3.5} directly without referring 
to Proposition~\ref{T3.6}. We will prove Proposition~\ref{T3.6} in Section~\ref{4}, after the proofs of Theorems~\ref{T3} 
and~\ref{T3.5} since it will not be used in these proofs.

We believe that the content of this paper will be of interest mostly for the community of people working with 
geometric aspects of Sobolev spaces. Since many of the researchers working in this area do not use tools 
form harmonic analysis, we decided to make the paper self-contained and easy to read by providing all
necessary details. But we also hope that researchers whose main area of research is harmonic analysis will find this paper
interesting too.

Notation used in the paper is pretty standard. The Fourier transform is defined by
$$
\hat{f}(\xi)=\int_{\bbbr^n} e^{-2\pi ix\cdot\xi}f(x)\, dx.
$$
By $C$ we will denote a positive constant whose value may change in a single string of estimates.

The paper is organized as follows.
In Section~\ref{2} we prove Lemmas~\ref{T4} and~\ref{T3.51}. The proofs are very elementary.
In Section~\ref{3} we prove Theorems~\ref{T3} and~\ref{T3.5}. The proofs use some harmonic analysis
including Theorems~\ref{T1.1} and~\ref{T6}. 
In Section~\ref{3a} we prove the second main result of the paper, Theorem~\ref{main2}
which is a generalization of Theorems~\ref{T3} and~\ref{T3.5} to the case of weighted Sobolev spaces
with a Muckenhoupt weight. Theorems~\ref{T3} and~\ref{T3.5} are special cases of Theorem~\ref{main2}, but we decided to 
include separate proofs in the unweighted case, because the proofs are based on more elementary arguments
(in particular we could use classical Theorems~\ref{T1.1} and~\ref{T6} in place of a more complicated Theorem~\ref{T11})
and the proofs of Theorems~\ref{T3} and~\ref{T3.5} are in fact used in the proof of Theorem~\ref{main2}.
Proposition~\ref{T3.6} which gives an inequality between $Sf$ and $Tf$
is presented in Section~\ref{4}. This result is not needed in the proofs of Theorems~\ref{T3}, \ref{T3.5}, and~\ref{main2}.
In Section~\ref{5} we include final remarks which are of independent interest - they are not needed in the proofs 
of the results in the earlier sections.

\noindent
{\bf Acknowledgments.} We would like to thank Yibiao Pan and Joan Verdera for helpful discussions.
Yibiao Pan showed us that the duality argument works also in the weighted case which was needed to complete the proof of
Theorem~\ref{main2}. We would also like to thank Shuichi Sato for providing a copy of his very recent work \cite{sato2}, where 
some of the results of this paper have been generalized.

\section{Proof of Lemmas~\ref{T4} and~\ref{T3.51}}
\label{2}

Both of the lemmas follow immediately from the lemma below.
\begin{lemma}
\label{T5}
If $f\in W^{1,1}_{\rm loc}(\bbbr^n)$, then
for all $t>0$ and almost all $x\in\bbbr^n$ we have
$$
f(x)-f_{S(x,t)}=
\frac{1}{n\omega_n}\int_{B(x,t)}\nabla f(y)\cdot\frac{x-y}{|x-y|^n}\, dy,
$$
$$
f(x)-f_{B(x,t)} =
\frac{1}{n\omega_n}\int_{B(x,t)}\nabla f(y)\cdot\frac{x-y}{|x-y|^n}\, dy -
\frac{1}{n}\,\barint_{B(x,t)}\nabla f(y)\cdot(x-y)\, dy.
$$
\end{lemma}
\begin{proof}
We can assume that $f\in C_0^\infty(\bbbr^n)$. The general case will follow by approximation.
Note that the restriction (trace) of $f\in W^{1,1}_{\rm loc}$ to $S(x,t)$ is well defined
\cite[Section~4.3, Theorem~1]{EG}.
We have
\begin{eqnarray*}
f(x)-\barint_{S(x,t)} f(y)d\sigma(y)
& = & 
-\int_0^t\frac{d}{d\tau}\Big(\barint_{S(x,\tau)}f(y)\, d\sigma(y)\Big)\, d\tau \\
& = &
-\int_0^t \frac{d}{d\tau}\Big(\barint_{S(0,1)} f(x+\tau z)\, d\sigma(z)\Big)\, d\tau \\
& = &
-\int_0^t\barint_{S(0,1)} \nabla f(x+\tau z)\cdot z\, d\sigma(z)\, d\tau \\
& = &
-\int_0^t \barint_{S(x,\tau)} \nabla f(y)\cdot\frac{y-x}{|y-x|}\, d\sigma(y)\, d\tau \\
& = &
\frac{1}{n\omega_n}\int_0^t\int_{S(x,\tau)}\nabla f(y)\cdot\frac{x-y}{|x-y|^n}\, d\sigma(y)\, d\tau \\
& = &
\frac{1}{n\omega_n}\int_{B(x,t)}\nabla f(y)\cdot\frac{x-y}{|x-y|^n}\, dy.
\end{eqnarray*}
This proves the first identity. The proof of the second one is similar.
\begin{eqnarray*}
f(x)-\barint_{B(x,t)}f(y)\, dy 
& = &
-\int_0^t\frac{d}{d\tau}\Big(\barint_{B(x,\tau)}f(y)\, dy\Big)\, d\tau \\
& = &
-\int_0^t\frac{d}{d\tau}\Big(\barint_{B(0,1)}f(x+\tau z)\ dz\Big)\, d\tau \\
& = &
-\int_0^t\barint_{B(0,1)}\nabla f(x+\tau z)\cdot z\, dz\, d\tau \\
& = &
-\int_0^t\barint_{B(x,\tau)}\nabla f(y)\cdot\frac{y-x}{\tau}\, dy\, d\tau \\
& = &
\frac{1}{\omega_n}\int_0^t\int_{B(x,\tau)}\nabla f(y)\cdot\frac{x-y}{\tau^{n+1}}\, dy\, d\tau \\
& = &
\frac{1}{\omega_n}\int_{B(x,t)}\Big(\int_{|x-y|}^t\frac{d\tau}{\tau^{n+1}}\Big)\nabla f(y)\cdot(x-y)\, dy \\
& = &
\frac{1}{n\omega_n}\int_{B(x,t)}\Big(\frac{1}{|x-y|^n}-\frac{1}{t^n}\Big)\nabla f(y)\cdot(x-y)\, dy.
\end{eqnarray*}
The proof is complete.
\end{proof}

\section{Proof of Theorems~\ref{T3} and~\ref{T3.5}}
\label{3}

\noindent
{\bf Step 1.} 
$\Vert Tf\Vert_p\leq C\Vert\nabla f\Vert_p$ and $\Vert Sf\Vert_p\leq C\Vert\nabla f\Vert_p$.

As we already pointed out the inequality $\Vert Tf\Vert_p\leq C\Vert\nabla f\Vert_p$ for $f\in W^{1,p}$ follows directly from
Theorem~\ref{T1.1} and Lemma~\ref{T4}. To prove the inequality $\Vert Sf\Vert_p\leq C\Vert\nabla f\Vert_p$ observe that
\begin{equation}
\label{e23}
\psi(x)=\phi(x)-\eta(x),
\quad
\mbox{where}
\quad
\eta(x)=\frac{x}{n\omega_n}\chi_{B(0,1)}.
\end{equation}
Hence
$$
Sf(x)\leq Tf(x)+ Wf(x),
$$
where
$$
Wf(x)=\Big(\int_0^\infty \big|\eta_t*\nabla f(x)\big|^2\frac{dt}{t}\Big)^{1/2}.
$$
It remains to show that $\Vert Wf\Vert_p\leq C\Vert\nabla f\Vert_p$, $f\in W^{1,p}$. This is however,
a consequence of Theorem~\ref{T6}. Indeed, the function $\eta$ clearly satisfies 
\eqref{e4} and \eqref{e5} with $\alpha=1$ and condition \eqref{e6} also holds with $\alpha=1$
which can be justified as follows.

If $|h|>1/2$, then
$$
\int_{\bbbr^n}|\eta(x+h)-\eta(x)|\, dx \leq 2\Vert\eta\Vert_1\leq C|h|.
$$

If $|h|\leq 1/2$, then
$$
n\omega_n \int_{\bbbr^n}|\eta(x+h)-\eta(x)|\, dx \leq
\int_{B(0,1+|h|)\setminus B(0,1-|h|)} 2\, dx +
\int_{B(0,1-|h|)} |h|\, dx\leq C|h|.
$$

\noindent
{\bf Step 2.} Square functions $\tilde{T}g$ and $\tilde{S}g$.

In this subsection we modify the definitions of the square functions $Tf$ and $Sf$
and prove boundedness of these modified square functions in $L^p$. It will play a crucial role in the 
proof of the reverse inequalities $\Vert \nabla f\Vert_p\leq \Vert Tf\Vert_p$ and 
$\Vert \nabla f\Vert_p\leq \Vert Sf\Vert_p$.

For $g\in L^p(\bbbr^n)$ let
$$
Rg=(R_1g,\ldots,R_ng)
$$
be the vector valued Riesz transform, where
$$
(R\vi)^\wedge(\xi) =-i\frac{\xi}{|\xi|}\hat{\vi}(\xi)
\quad
\mbox{for $\vi\in\mathscr{S}(\bbbr^n)$.}
$$
Now we define
$$
\tilde{T}g(x)=\Big(\int_0^\infty |\phi_t*Rg(x)|^2\frac{dt}{t}\Big)^{1/2},
\quad
g\in L^p(\bbbr^n),
$$
$$
\tilde{S}g(x)=\Big(\int_0^\infty |\psi_t*Rg(x)|^2\frac{dt}{t}\Big)^{1/2},
\quad
g\in L^p(\bbbr^n).
$$
\begin{lemma}
\label{T13}
For $1<p<\infty$ we have
\begin{equation}
\label{e15}
\Vert\tilde{T}g\Vert_p\leq C\Vert g\Vert_p,
\quad
g\in L^p(\bbbr^n),
\end{equation}
\begin{equation}
\label{e16}
\Vert\tilde{S}g\Vert_p\leq C\Vert g\Vert_p,
\quad
g\in L^p(\bbbr^n).
\end{equation}
\end{lemma}
\begin{proof}
Estimate \eqref{e15} follows from Theorem~\ref{T1.1} and boundedness of the Riesz transform in $L^p$ while
\eqref{e16} follows from
$$
\tilde{S}g\leq \tilde{T}g + \Big(\int_0^\infty |\eta_t*Rg|^2\frac{dt}{t}\Big)^{1/2}
$$
combined with \eqref{e15}, the fact that $\eta$ satisfies the assumptions of Theorem~\ref{T6} and from
boundedness of the Riesz transform in $L^p$.
\end{proof}

\noindent
{\bf Step 3.} $L^2$ isometries.

We will prove that up to a constant factor, the square functions $\tilde{T}$ and $\tilde{S}$ are isometries in $L^2$, i.e.
\begin{lemma}
\label{T14}
There are constants $C_1,C_2>0$ such that
$$
\Vert \tilde{T}g\Vert_2=C_1\Vert g\Vert_2
\quad
\mbox{and}
\quad
\Vert \tilde{S}g\Vert_2=C_2\Vert g\Vert_2
\quad
\mbox{for $g\in L^2(\bbbr^n)$.}
$$
\end{lemma}
Observe that the functions $\phi$ and $\psi$ are of the form $g(|x|)x/|x|$. This allows us to find the structure
of the Fourier transforms of $\phi$ and $\psi$. 
\begin{lemma}
\label{T10}
Let $f\in L^1(\bbbr^n,\bbbr^n)$ be of the form
$$
f(x)=\frac{x}{|x|}g(|x|).
$$
Then there is a continuous function $h:[0,\infty)\to\bbbr$, $h(0)=0$, $h(t)\to 0$ as $t\to\infty$
such that
$$
\hat{f}(\xi)=i\, \frac{\xi}{|\xi|}h(|\xi|),
\quad
\xi\in\bbbr^n.
$$
\end{lemma}
The proof is based on the following well known result from linear algebra \cite[Lemma~4.1.15]{grafakos-old}, \cite[p.57]{stein3}.
\begin{lemma}
\label{T9}
If $m:\bbbr^n\to\bbbr^n$ is a measurable function that is homogeneous of degree $0$, i.e.
$m(tx)=m(x)$ for $t>0$, and commutes with orthogonal transformations, i.e.
$$
m(\rho(x))=\rho(m(x)),
\quad
x\in\bbbr^n,\ \rho\in O(n),
$$
then there is a constant $C\in\bbbr$ such that
$$
m(x)=C\frac{x}{|x|}
\quad
\mbox{for all $x\neq 0$.}
$$
\end{lemma}
\begin{proof}[Proof of Lemma~\ref{T10}]
Since the function $f$ is odd and takes values in $\bbbr^n$, the real part of $\hat{f}$ equals zero, and hence
$$
i\hat{f}(\xi)=\int_{\bbbr^n} \sin(2\pi x\cdot\xi)f(x)\, dx
$$
takes values in $\bbbr^n$.
Fix $k>0$ and define $m_k:S^{n-1}(0,k)\to\bbbr^n$ by
$m_k(\xi)=i\hat{f}(\xi)$ for $|\xi|=k$. Extend $m_k$ to
$m_k:\bbbr^n\setminus\{ 0\}\to\bbbr^n$ as a function homogeneous of degree $0$, i.e.
$$
m_k(\xi)=i\, \hat{f}\Big(\frac{k\xi}{|\xi|}\Big)
\quad
\mbox{for $\xi\neq 0$.}
$$
We claim that
\begin{equation}
\label{e10}
m_k(\rho(\xi))=\rho(m_k(\xi))
\quad
\mbox{for $\rho\in O(n)$ and  $\xi\neq 0$.}
\end{equation}
Indeed, it suffices to check \eqref{e10} for $|\xi|=k$. We have
\begin{eqnarray*}
m_k(\rho(\xi))
& = & 
\int_{\bbbr^n}\sin(2\pi x\cdot\rho(\xi))\frac{x}{|x|}g(|x|)\, dx
=
\int_{\bbbr^n}\sin(2\pi \rho^{-1}(x)\cdot\xi)\frac{x}{|x|}g(|x|)\, dx\\
& = &
\int_{\bbbr^n}\sin(2\pi x\cdot\xi)\frac{\rho(x)}{|\rho(x)|}g(|\rho(x)|)\, dx 
=
\int_{\bbbr^n}\sin(2\pi x\cdot\xi)\frac{\rho(x)}{|x|}g(|x|)\, dx \\
& = &
\rho\Big(\int_{\bbbr^n}\sin(2\pi x\cdot\xi)\frac{x}{|x|}g(|x|)\, dx\Big)
=
\rho(m_k(\xi)).
\end{eqnarray*}
According to Lemma~\ref{T9} there is a constant $h(k)\in\bbbr$ such that
$m_k(\xi)=-h(k)\xi/|\xi|$. In particular for $|\xi|=k$ we have
$$
i\hat{f}(\xi)=m_k(\xi)=-\frac{\xi}{|\xi|}h(|\xi|).
$$
Clearly $h$ is continuous, $h(0)=0$ and $h(t)\to 0$ as $t\to\infty$,
because $\hat{f}\in C_0(\bbbr^n,\bbbr^n)$.
\end{proof}

\begin{proof}[Proof of Lemma~\ref{T14}]
We will prove the result in the case of the square function $\tilde{T}g$ only.
The proof in the case of $\tilde{S}g$ is the same.
Using the Fubini theorem, the Plancherel theorem and the fact that 
$\hat{\phi}_t(\xi)=\hat{\phi}(t\xi)$ we obtain
$$
\Vert\tilde{T}g\Vert_2^2=
\int_0^\infty\int_{\bbbr^n}|\phi_t*Rg|^2\, dx\, \frac{dt}{t} =
\int_0^\infty\int_{\bbbr^n} |\hat{\phi}(t\xi)\cdot\widehat{Rg}(\xi)|^2\, d\xi\frac{dt}{t} =
\heartsuit,
$$
$$
\hat{\phi}(t\xi)=i\frac{\xi}{|\xi|}h(t|\xi|),
\quad
\widehat{Rg}(\xi) = -i\frac{\xi}{|\xi|}\hat{g}(\xi),
$$
$$
\heartsuit =
\int_0^\infty\int_{\bbbr^n}|\hat{g}(\xi)h(t|\xi|)|^2\, d\xi\frac{dt}{t} =
\int_{\bbbr^n}|g(x)|^2\, dx \int_0^\infty|h(t)|^2\frac{dt}{t} =
C\Vert g\Vert_2^2.
$$
Note that the integral involving $h$ does not depend on $|\xi|$
(use the change of variables $s=t|\xi|$).
Since the square function $\tilde{T}g$ is bounded in $L^2$ we conclude that
$$
\int_0^\infty|h(t)|^2\frac{dt}{t} =C<\infty.
$$
The proof is complete.
\end{proof}

\noindent
{\bf Step 4.} Duality argument.

We will use a standard duality argument \cite[Remark~5.6, p.~507]{GC}, \cite[Exercise~5.1.6]{grafakos},
to show that Lemma~\ref{T13} and Lemma~\ref{T14} imply
\begin{lemma}
\label{T15}
For $1<p<\infty$ there is a constant $C\geq 1$ such that
$$
C^{-1}\Vert g\Vert_p\leq \Vert\tilde{T}g\Vert_p\leq C\Vert g\Vert_p,
\quad
g\in L^p(\bbbr^n),
$$
$$
C^{-1}\Vert g\Vert_p\leq \Vert\tilde{S}g\Vert_p\leq C\Vert g\Vert_p,
\quad
g\in L^p(\bbbr^n).
$$
\end{lemma}
\begin{proof}
We will prove the result in the $\tilde{T}g$ case, the proof in the $\tilde{S}g$ case is the same.
Consider the following operator acting on functions $g$ defined on $\bbbr^n$ whose values at $x\in\bbbr^n$ are measurable functions
of variable $t\in\bbbr_+$
$$
Kg(x)=\big(\phi_t*Rg(x)\big)_{t>0}.
$$
Lemma~\ref{T13} states that $K$ is a bounded operator between the spaces
\begin{equation}
\label{e17}
K:L^p(\bbbr^n)\to L^p(\bbbr^n,L^2(\bbbr_+,dt/t))=L^p(\bbbr^n,H),
\end{equation}
and Lemma~\ref{T14} means that
$K:L^2(\bbbr^n)\to L^2(\bbbr^n,H)$ is an isometry multiplied by a constant factor
\begin{equation}
\label{e18}
\Big(\int_{\bbbr^n}\Vert Kg\Vert_H^2\, dx\Big)^{1/2} =
C_1\Big(\int_{\bbbr^n}|g|^2\, dx\Big)^{1/2},
\quad
g\in L^2(\bbbr^n).
\end{equation}
Here we consider real valued functions $g$. Let $q$ be the H\"older conjugate exponent to $p$, $p^{-1}+q^{-1}=1$.
Since the scalar product is determined by the Hilbert norm (polarization identity)
we conclude from \eqref{e18} and from \eqref{e15} with $p$ replaced by $q$ that for
$g\in L^p\cap L^2$ and $h\in L^q\cap L^2$ we have
\begin{eqnarray*}
C_1^2\int_{\bbbr^n} gh\, dx
& = &
\int_{\bbbr^n}\langle Kg,Kh\rangle_H\, dx
\leq
\Big(\int_{\bbbr^n}\Vert Kg\Vert_H^p\, dx\Big)^{1/p}
\Big(\int_{\bbbr^n}\Vert Kh\Vert_H^q\, dx\Big)^{1/q} \\
& = &
\Vert\tilde{T}g\Vert_p\Vert\tilde{T}h\Vert_q
\leq
C\Vert\tilde{T}g\Vert_p\Vert h\Vert_q.
\end{eqnarray*}
Taking supremum over $h\in L^q\cap L^2$, $\Vert h\Vert_q\leq 1$ we obtain
$\Vert g\Vert_p\leq C\Vert\tilde{T}g\Vert_p$.
We proved this inequality for $g\in L^p\cap L^2$, but a density argument shows that it is true for any $g\in L^p(\bbbr^n)$.
\end{proof}

\noindent
{\bf Step 5.} Fractional Laplacian $(-\Delta)^{1/2}.$

In this section we will prove the left inequalities at \eqref{e19} and \eqref{e20} for $f\in W^{1,p}$
by applying Lemma~\ref{T15} to $g=(-\Delta)^{1/2}f$. 

Recall that the fractional Laplace operator is defined by
$$
(-\Delta)^{1/2}\vi = (2\pi|\xi|\hat{\vi}(\xi))^\vee,
\qquad
\vi\in \mathscr{S}(\bbbr^n).
$$
\begin{lemma}
\label{T16}
For $1<p<\infty$ there is a constant $C\geq 1$ such that
$$
C^{-1}\Vert\nabla\vi\Vert_p\leq\Vert(-\Delta)^{1/2}\vi\Vert_p\leq C\Vert\nabla\vi\Vert_p,
\qquad
\vi\in\mathscr{S}(\bbbr^n).
$$
\end{lemma}
\begin{proof}
By taking the Fourier transform and looking at the multipliers we see that
$$
R(-\Delta)^{1/2}\vi = -\nabla\vi
\quad
\mbox{and}
\quad
R\cdot\nabla\vi = (-\Delta)^{1/2}\vi,
$$
where $R\cdot\nabla\vi = \sum_j R_j\partial_j\vi$. Then the result follows from boundedness of
$R$ in $L^p$.
\end{proof}
By continuity $(-\Delta)^{1/2}$ uniquely extends to a bounded operator
$$
(-\Delta)^{1/2}:W^{1,p}(\bbbr^n)\to L^p(\bbbr^n), 
\quad
1<p<\infty
$$
that also satisfies the inequality of Lemma~\ref{T16} and
$$
R(-\Delta)^{1/2}f=-\nabla f
\quad
\mbox{for $f\in W^{1,p}(\bbbr^n)$, $1<p<\infty$.}
$$
Now for $f\in W^{1,p}(\bbbr^n)$ Lemma~\ref{T15} yields
$$
\Vert Tf\Vert_p=\Vert\tilde{T}\big((-\Delta)^{1/2}f\big)\Vert_p\approx
\Vert(-\Delta)^{1/2}f\Vert_p\approx\Vert\nabla f\Vert_p
$$
and similarly $\Vert Sf\Vert_p\approx\Vert\nabla f\Vert_p$.

\noindent
{\bf Step 6.} The final step of the proof.

We proved that $f\in W^{1,p}$, $1<p<\infty$ satisfies the inequalities \eqref{e19} and \eqref{e20}. It remains to prove that
if $f\in L^p$ satisfies $Tf\in L^p(\bbbr^n)$ or $Sf\in L^p(\bbbr^n)$, then $f\in W^{1,p}$.

Let $f\in L^p(\bbbr^n)$, $1<p<\infty$ and assume that $Tf\in L^p$ or $Sf\in L^p$.
Observe that the functions $Tf$ and $Sf$ are of the form
$$
Kf(x)=\Big(\int_0^\infty |f*\mu_t(x)|^2\frac{dt}{t^3}\Big)^{1/2},
$$
where $\mu_t$ is some measure. For example in the case of the square function $Tf$, 
$\mu_t$ is the Dirac delta minus the Hausdorff measure $\H^{n-1}$ restricted to $S^{n-1}(0,t)$,
normalized to have total mass $1$.

Let $\vi$ be a standard mollifier, i.e. $\vi\in C_0^\infty(\bbbr^n)$, $\vi\geq 0$,
$\int_{\bbbr^n}\vi(x)\, dx =1$. Let $\vi_\eps(x)=\eps^{-n}\vi(x/\eps)$ and 
$f_\eps=f*\vi_\eps$. Clearly $f_\eps\to f$ in $L^p$. Since 
$\nabla f_\eps=f*\nabla\vi_\eps$, $\Vert\nabla f_\eps\Vert_p\leq\Vert f\Vert_p\, \Vert\nabla\vi_\eps\Vert_1<\infty$.
Hence $f_\eps\in W^{1,p}$ and thus
$$
\Vert Kf_\eps\Vert_p\approx\Vert\nabla f_\eps\Vert_p.
$$
On the other hand for almost all $x\in\bbbr^n$ we have
$$
Kf_\eps(x)
=
\Big(\int_0^\infty\Big|\int_{\bbbr^n}(f*\mu_t)(x-y)\vi_\eps(y)\, dy\Big|^2\frac{dt}{t^3}\Big)^{1/2}
\leq
\Big(\int_0^\infty\Big(\int_{\bbbr^n}|F(y,t)|\vi_\eps(y)\, dy\Big)^{2}\frac{dt}{t^3}\Big)^{1/2},
$$
where $F(y,t)=(f*\mu_t)(x-y)$.
Thus the Minkowski integral inequality, \cite{stein3}, implies
$$
Kf_\eps(x)
\leq
\int_{\bbbr^n}\Big(\int_0^\infty|F(y,t)|^2\frac{dt}{t^3}\Big)^{1/2}\vi_\eps(y)\, dy
=
Kf*\vi_\eps(x).
$$
Hence
$$
\Vert\nabla f_\eps\Vert_p\approx\Vert Kf_\eps\Vert_p\leq \Vert Kf\Vert_p\Vert\vi_\eps\Vert_1 = \Vert Kf\Vert_p<\infty.
$$
Since the functions $\nabla f_\eps$ are bounded in $L^p$, there is a sequence $\eps_k\downarrow 0$ such that
the sequence $\nabla f_{\eps_k}$ converges weakly in $L^p$. Since $f_{\eps_k}\to f$ in $L^p$ we conclude that
$f\in W^{1,p}$. The proof is complete.

\section{Weighted Sobolev spaces}
\label{3a}

In this section we will show that the claims of Theorems~\ref{T3} and~\ref{T3.5} remain valid in the weighted Sobolev space
$W^{1,p}_w(\bbbr^n)$, $1<p<\infty$, where $w\in A_p$ as a Muckenhoupt weight; see \cite{GC,grafakos2} for the theory of Muckenhoupt 
weights and \cite[Chapter 1]{HKM} for the theory of weighted Sobolev spaces.
We will write $L^p_w$ or $L^p(w)$ to denote the weighted $L^p$ space.

\begin{theorem}
\label{main2}
Let $w\in A_p$, $1<p<\infty$, and let $f\in L^p_w(\bbbr^n)$.
Then the following conditions are equivalent
\begin{enumerate}
\item $f\in W^{1,p}_w(\bbbr^n)$,
\item $Tf\in L^p_w(\bbbr^n)$,
\item $Sf\in L^p_w(\bbbr^n)$.
\end{enumerate}
Moreover for $f\in W^{1,p}_w(\bbbr^n)$ we have
$\Vert Tf\Vert_{L^p_w}\approx \Vert Sf\Vert_{L^p_w} \approx \Vert\nabla f\Vert_{L^p_w}$.
\end{theorem}
\begin{proof}
First of all observe that for $w\in A_p$, $W^{1,p}_w(\bbbr^n)\subset W^{1,1}_{\rm loc}(\bbbr^n)$ (see 
\cite[p.14]{HKM}) and hence Lemmas~\ref{T4} and~\ref{T3.51} remain valid in the weighted Sobolev space $W^{1,p}_w(\bbbr^n)$.

The proof is based on the following result of Sato \cite[Corollary~1]{sato}
which is a generalization of Theorem~\ref{T1.1}.
\begin{theorem}
\label{T11}
For $\eps>0$ let
$$
\zeta(x)=|x|^{-n+\eps}\Omega(x')\chi_{B(0,1)}(x),
$$
where $x'=x/|x|$, $\Omega\in L^\infty(S^{n-1})$, and $\int_{S^{n-1}}\Omega(x)\, d\sigma(x)=0$. Then the square function
$$
\sigma(f)(x)=
\Big(\int_0^\infty |\zeta_t*f(x)|^2\frac{dt}{t}\Big)^{1/2}
$$
is bounded in the weighted space $L^p_w(\bbbr^n)$, $w\in A_p$, $1<p<\infty$,
$$
\Vert\sigma(f)\Vert_{L^p_w}\leq C_{p,w}\Vert f\Vert_{L^p_w}.
$$
\end{theorem}
As an immediate application we obtain
that if $g\in L^p_w(\bbbr^n)$, then
\begin{equation}
\label{efef}
\Vert\tilde{T}g\Vert_{L^p_w}\leq C\Vert g\Vert_{L^p_w},
\quad
\Vert\tilde{S}g\Vert_{L^p_w}\leq C\Vert g\Vert_{L^p_w}.
\end{equation}
Indeed, the functions
$\phi$ and $\eta$ (defined in \eqref{e23}) satisfy the assumptions of Theorem~\ref{T11}.
Since $\psi=\phi-\eta$ (see \eqref{e23}) boundedness of the Riesz transform in $L^p_w$ (see \cite[Chapter~IV.3]{GC},\cite{grafakos2}) yields \eqref{efef}.
It is also easy to see that the duality argument works in the weighted case too. Indeed, let $q$ be the H\"older conjugate
exponent to $p$. It directly follows from the definition of the Muckenhoupt weight that $w^{-q/p}\in A_q$.
Hence the unweighted isometry, Lemma~\ref{T14} yields
\begin{eqnarray*}
C_1^2 \int_{\bbbr^n} gh\, dx
& \leq &
\int_{\bbbr^n} \Vert Kg\Vert_H\Vert Kh\Vert_H\, dx \\
& \leq &
\Big( \int_{\bbbr^n}\Vert Kg\Vert_H^p w\, dx\Big)^{1/p}
\Big( \int_{\bbbr^n}\Vert Kh\Vert_H^q w^{-q/p}\, dx\Big)^{1/p}\\
& = &
\Vert\tilde{T}g\Vert_{L^p_w}\Vert\tilde{T}h\Vert_{L^q(w^{-q/p})} 
 \leq 
C\Vert\tilde{T}g\Vert_{L^p_w}\Vert h\Vert_{L^q(w^{-q/p})}.
\end{eqnarray*} 
Taking the supremum over $\Vert h\Vert_{L^q(w^{-q/p})}\leq 1$ we get
\begin{eqnarray*}
\Vert g\Vert_{L^p_w}
& = &
\Vert gw^{1/p}\Vert_p =
\sup_{\Vert hw^{-1/p}\Vert_q\leq 1} \int_{\bbbr^n} gw^{1/p}hw^{-1/p}\, dx \\
& = &
\sup_{\Vert h\Vert_{L^q(w^{-q/p})}\leq 1} \int_{\bbbr^n} gh \leq C\Vert\tilde{T}g\Vert_{L^p_w}.
\end{eqnarray*}
A similar argument applies also to $\tilde{S}g$. This combined with \eqref{efef} yields that
$$
\Vert\tilde{T}g\Vert_{L^p_w}\approx\Vert \tilde{S}g\Vert_{L^p_w}\approx\Vert g\Vert_{L^p_w}.
$$
Since the Riesz transform is bounded in $L^p_w$, Lemma~\ref{T16} is also true if we replace $L^p$ by 
$L^p_w$. Hence as in the unweighted case for $f\in W^{1,p}_w$ we have
$$
\Vert Tf\Vert_{L^p_w}\approx \Vert Sf\Vert_{L^p_w}\approx \Vert\nabla f\Vert_{L^p_w}.
$$
Finally the argument used to show that if $f\in L^p$ and $Tf\in L^p$ or $Sf\in L^p$, then
$f\in W^{1,p}$ also extends to the weighted case, but we need to use boundedness of the Hardy-Littlewood maximal function
$\M$ in $L^p_w$ (see \cite{GC},\cite[Theorem~9.1.9]{grafakos2}).

Suppose that $f\in L^p_w$ and $Tf\in L^p_w$ or $Sf\in L^p_w$. We need to show that $f\in W^{1,p}_w$. Since
$L^p_w\subset L^1_{\rm loc}$, $f_\eps=f*\vi_\eps\to f$ a.e. It is easy to see that
$|f_\eps|\leq C\M f\in L^p_w$, and hence the dominated convergence theorem yields that $f_\eps\to f$ in $L^p_w$.
Moreover
$$
|\nabla f_\eps|=|f*\nabla\vi_\eps|\leq C\eps^{-1}\M f\in L^p_w,
$$
so $f_\eps\in W^{1,p}_w$. Accordingly $\Vert Kf_\eps\Vert_{L^p_w}\approx \Vert\nabla f\Vert_{L^p_w}$. From
the proof in the unweighted case we have
$$
Kf_\eps\leq Kf*\vi_\eps\leq C\M(Kf)\in L^p_w,
$$
$$
\Vert\nabla f_\eps\Vert_{L^p_w}\approx \Vert Kf_\eps\Vert_{L^p_w}\leq C\Vert \M(Kf)\Vert_{L^p_w}<\infty.
$$
Thus $f_\eps$ is a bounded family in the reflexive Sobolev space $W^{1,p}_w$ (see \cite[p.13]{HKM}
for the proof of reflexivity). Since $f_\eps\to f$ in $L^p_w$, a weak compactness argument implies that
$f\in W^{1,p}_w(\bbbr^n)$. The proof is complete.
\end{proof}

\section{Proof of Proposition~\ref{T3.6}}
\label{4}

Integration by parts gives
\begin{eqnarray*}
\omega_n^2 \int_\eps^T |f(x)-f_{B(x,t)}|^2\, \frac{dt}{t^3} 
& = &
\int_\eps^T \Big(\int_{B(x,t)}(f(x)-f(y))\, dy\Big)^2\frac{dt}{t^{2n+3}} \\
& = &
\left.\frac{t^{-2n-2}}{-2n-2} \Big(\int_{B(x,t)}(f(x)-f(y))\, dy\Big)^2\right|_\eps^T \\
& + & 
\frac{1}{2n+2}\int_\eps^T t^{-2n-2}\frac{d}{dt}\Big(\int_{B(x,t)}(f(x)-f(y))\, dy\Big)^2\, dt \\
& = &
A(t)\big|_\eps^T +B_{\eps,T}.
\end{eqnarray*}
We have
\begin{eqnarray*}
B_{\eps,T}
& = &
\frac{1}{2n+2} \int_\eps^T t^{-2n-2}\cdot 2
\Big(\int_{B(x,t)}(f(x)-f(y))\, dy\Big)
\Big(\int_{S(x,t)}(f(x)-f(y))\, d\sigma(y)\Big)\, dt \\
& = &
\frac{2n\omega_n^2}{2n+2} \int_\eps^T
(f(x)-f_{B(x,t)})(f(x)-f_{S(x,t)})\, \frac{dt}{t^3} \\
& \leq &
\frac{n\omega_n^2}{2n+2}
\Big[
\int_\eps^T |f(x)-f_{B(x,t)}|^2\frac{dt}{t^3} +
\int_\eps^T |f(x)-f_{S(x,t)}|^2\frac{dt}{t^3}
\Big]\, .
\end{eqnarray*}
In the last step we applied the inequality $ab\leq (a^2+b^2)/2$.
Thus
$$
\omega_n^2\Big(1-\frac{n}{2n+2}\Big) \int_\eps^T |f(x)-f_{B(x,t)}|^2\frac{dt}{t^3} 
\leq
A(t)\big|_\eps^T + 
\frac{n\omega_n^2}{2n+2} 
\int_\eps^T |f(x)-f_{S(x,t)}|^2\frac{dt}{t^3}
$$
and it suffices to prove that for almost all $x$,
$$
A(t)\big|_\eps^T = A(T)-A(\eps)\to 0
\quad
\mbox{as $T\to\infty$ and $\eps\to 0$.}
$$
We have
$$
A(T)=C(n)\Big(\frac{1}{T}\big(f(x)-\barint_{B(x,T)}f(y)\, dy\big)\Big)^2\to 0
\quad
\mbox{as $T\to\infty$ for almost all $x\in\bbbr^n$.}
$$
Since
$$
\int_{B(x,\eps)}\nabla f(x)\cdot (y-x)\, dy = 0
$$
we have
$$
A(\eps) = C(n)\Big(\frac{1}{\eps}\,\barint_{B(x,\eps)} \big(f(x)+\nabla f(x)\cdot(y-x)-f(y)\big)\, dy\Big)^2 \to 0
\quad
\mbox{as $\eps\to 0$}
$$
for almost all $x\in\bbbr^n$ by
\cite[Section~6.1.2, Theorem~2]{EG}.
\hfill $\Box$

\section{Final remarks and comments}
\label{5}

\subsection{Sobolev spaces on metric spaces}
As was pointed out by Alabern, Mateu and Verdera \cite{AMV}, the characterization of $W^{1,p}$ given in Theorem~\ref{T3.5}
can be used to define a Sobolev space on any metric-measure space. It is an interesting question to see how this definition
is related to other definitions existing in the literature, \cite{cheeger,hajlasz2,hajlasz,SmP,shanmugalingam}. 
Formally one could try to use the characterization given in Theorem~\ref{T3}
to define a Sobolev space on a metric-measure space, but the main difficulty would be that, in general, there is no reasonable way to define 
measure on the boundary of a ball which would be needed for the spherical averages. It may even happen that the boundary of a ball is empty.

\subsection{The spherical maximal function}
Kinnunen \cite{kinnunen} proved that the Hardy-Littlewood maximal function
is bounded in the Sobolev space, $\M:W^{1,p}(\bbbr^n)\to W^{1,p}(\bbbr^n)$, $1<p<\infty$.
Actually, any sub-linear operator that commutes with translations and is bounded in $L^p$, $1<p<\infty$
is also bounded in $W^{1,p}$, see e.g. \cite[Theorem~1]{hajlaszo}. From this result it follows that the spherical maximal operator
$$
\mathscr{S} f(x)=\sup_{t>0} \Big|\barint_{S(x,t)} f(y)\, d\sigma(y)\Big|
$$
is bounded in $W^{1,p}$ for $p>n/(n-1)$. Indeed, according to a celebrated result of Stein \cite{stein4}, and Bourgain \cite{bourgain},
$\mathscr{S}:L^p(\bbbr^n)\to L^p(\bbbr^n)$ is bounded for $p>n/(n-1)$. It was conjectured in \cite{hajlaszo} that in the range $1<p\leq n/(n-1)$
the spherical maximal operator is a bounded operator from $W^{1,p}$ to the homogeneous Sobolev space $\dot{W}^{1,p}$;
see \cite{hajlaszl1,hajlaszl2} for results supporting this conjecture.

The next result which is a direct consequence of Lemma~\ref{T5} provides another support for this conjecture as it
allows to represent $\mathscr{S} f$ as a Hardy-Littlewood type maximal function.
\begin{theorem}
For $f\in W^{1,1}_{\rm loc}(\bbbr^n)$ we have
$$
\mathscr{S} f(x) = \sup_{t>0} \Big|\barint_{B(x,t)}\Big(f(y)-\frac{1}{n}\nabla f(y)\cdot(x-y)\Big)\, dy\Big|.
$$
\end{theorem}

\end{document}